\documentclass[11pt, twoside, a4paper]{article}
\usepackage[margin = 1in]{geometry}
\usepackage{amsmath, amssymb, amsfonts}
\usepackage{amsthm}
\usepackage[utf8]{inputenc} 
\usepackage[T1]{fontenc}    
\usepackage{amsthm}
\usepackage{hyperref}       
\usepackage{url}            
\usepackage{booktabs}       
\usepackage{amsfonts}       
\usepackage{nicefrac}       
\usepackage{microtype}      
\usepackage{comment}
\usepackage[pdftex]{graphicx}

\newcommand{\norm}[2][]{\ensuremath{\left\Vert #2 \right\Vert_{#1}}}
\newcommand{\spec}[1]{\ensuremath{\mathrm{sp}\inp{#1}}}

\newcommand{\inp}[1]{\left(#1\right)}

\makeatletter
\newcommand*{\defeq}{\mathrel{\rlap{%
                     \raisebox{0.3ex}{$\m@th\cdot$}}%
                     \raisebox{-0.3ex}{$\m@th\cdot$}}%
                    =}
\newcommand*{\eqdef}{=
  \mathrel{\rlap{%
      \raisebox{0.3ex}{$\m@th\cdot$}}%
    \raisebox{-0.3ex}{$\m@th\cdot$}}%
}
\makeatother

\newcommand{\R}[0]{\ensuremath{\mathbb{R}}}

\newcommand{\etal}[0]{\emph{et al. }}

\newcommand{\ie}[0]{\emph{i.e.}}

\renewcommand{\vec}[1]{\mathbf{#1}}

\newtheorem{theorem}{Theorem}

\newtheorem{lemma}[theorem]{Lemma}
\newtheorem{proposition}[theorem]{Proposition}
\newtheorem{corollary}[theorem]{Corollary}
\newtheorem{remark}[theorem]{Remark}

\newtheorem{definition}[theorem]{Definition}

\begin{document}
\title{\bf Gradient Descent Only Converges to Minimizers:\\ Non-Isolated Critical Points and Invariant Regions}

%

\author{Ioannis Panageas \\Georgia Institute of Technology\\ ioannis@gatech.edu
\and Georgios Piliouras \\ Singapore University of Technology \& Design  \\georgios@sutd.edu.sg}

\date{}
\maketitle
\thispagestyle{empty}

\begin{abstract}
Given a non-convex twice differentiable cost function $f$,
we prove that the set of initial conditions so that gradient descent converges to saddle points where $\nabla^2 f$ has at least one strictly negative eigenvalue has (Lebesgue) measure zero, even for cost functions $f$ with non-isolated critical points, answering an open question in \cite{DBLP:journals/corr/LeeSJR16}. Moreover, this result extends to forward-invariant convex subspaces, allowing for weak (non-globally Lipschitz) smoothness assumptions. Finally, we produce an upper bound on the allowable step-size.
\end{abstract}

\section{Introduction}

The interplay between the structure of saddle points and the performance of gradient descent dynamics is a critical and not well understood aspect of non-convex optimization.
Despite our incomplete theoretical understanding,
in practice, the intuitive nature of the gradient descent method (and more generally gradient-like algorithms\footnote{A gradient-like system is a system where for each non-equilibrium initial condition the dynamic will move towards a new state whose cost is strictly less than that of the initial state.}) make it a basic tool for attacking non-convex optimization problems for which we have very little understanding of the geometry of their saddle points. In fact, these techniques become particularly useful as the equilibrium structure becomes increasingly complicated, e.g., such as in
the cases of nonnegative matrix factorization \cite{lee2001algorithms} or congestion/potential games \cite{sandholm2009evolutionary}, where symmetries in the nature of non-convex optimization problems give rise to continuums of saddle points with complex geometry.
In these cases, especially, the simplistic, greedy attitude of the gradient descent method, which is by design agnostic towards the global geometry of the cost function minimized, comes rather handy. As we move forward in time, the cost keeps decreasing and convergence is guaranteed.

This simplicity, however, comes at least seemingly at a  significant cost. For example, it is well known that there exist instances where bad initialization of gradient descent converges to saddle points \cite{Nesterov04}. Despite the existence of such worst case instances in theory, practitioners have been rather successful at applying these techniques across a wide variety of problems \cite{ravindran2006engineering}.  Recently, Lee et. al. \cite{DBLP:journals/corr/LeeSJR16} have given a rather insightful interpretation of the effectiveness of gradient descent methods in terms of circumventing the saddle equilibrium problem using tools from topology of dynamical systems. At a glance, the paper argues the following intuitively clear message: The instability of  (locally unstable) saddle points translates to a global phenomenon and the probability of converging to such a saddle point given a randomly chosen (random not over a local neighborhood but over the whole state space) initial condition is zero.

This message is clear, concise, and satisfying in the sense that it transcribes the practical success of the gradient descent method to a concrete theoretical guarantee.  As is usually the case, such high level statements come with an asterisk of necessary technical conditions on the cost function $f$ minimized.

Formally, Lee et. al. define a cost function $f$ as satisfying the ``strict saddle'' property if each critical point\footnote{$x$ is a critical point of $f$ if $\nabla f(x) = 0$.}  $x$ of $f$ is either a local minimizer, or a strict saddle, i.e, $\nabla^2f(x)$ has at least one strictly negative eigenvalue. They argue that if  $f:\R^d \rightarrow \R$ is a twice continuously differentiable function then gradient descent with constant step-size $\alpha$ (defined by $x_{k+1} = x_{k} - \alpha \nabla f(x_k)$) with a random initialization and sufficiently small constant step-size converges to a local minimizer or negative infinity almost surely.

Critically, for this result to apply, $f$ is required to have isolated saddle points, $\nabla f$ is assumed to be globally $L$-Lipschitz\footnote{That is, $f$ satisfies $\norm[2]{\nabla f(\vec{x})- \nabla f(\vec{y})} \leq L \norm[2]{\vec{x}-\vec{y}}$.} and the step-size $\alpha$ is taken to be less than $1/L$. These regularity conditions soften somewhat the impact of the statement
both theoretically as well as in practice. First, although the assumption of isolated fixed points is indeed generic for abstract classes of cost functions, in several special cases of practical interest where the cost function has some degree of symmetry (e.g., due to scaling invariance) this assumption is not satisfied. For this reason, the important question of whether the assumption of isolated equilibria is indeed necessary was explicitly raised in \cite{DBLP:journals/corr/LeeSJR16}. Moreover, the assumption of global Lipschitz continuity for $\nabla f$ is not satisfied even by low degree polynomials (e.g., cubic). Finally, a natural question is how tight is the assumption on the step-size?

In this work we provide answers to all the above questions. We show that the assumption of isolated saddle points is indeed not necessary to argue generic convergence to local minima. To argue this, we need to combine tools from dynamical systems, topology, analysis and optimization theory. Moreover, we show that the globally Lipschitz assumption can be circumvented as long as the domain is convex and forward invariant with respect to gradient descent. This proposition makes our results easily applicable to many standard settings. Finally, using linear algebra and eigenvalue analysis we provide an upper bound on the allowable step-size (for these results to hold).
Our work shows that the high level message of \cite{DBLP:journals/corr/LeeSJR16} is effectively practically always binding. Saddle points are indeed of little concern for the gradient descent method in practice, but it takes quite a bit of theory to argue so.

\medskip
\subsection{Related work}\label{sec:related}

First-order descent methods can indeed escape strict saddle points when assisted by near isotropic noise.
\cite{pemantle1990nonconvergence} establishes convergence of the Robbins-Monro stochastic approximation to local minimizers for strict saddle functions, whereas \cite{kleinberg2009multiplicative} establishes convergence to local minima for perturbed versions of multiplicative weights algorithm in generic potential games. Recently, \cite{ge2015escaping} quantified  the convergence rate of perturbed stochastic gradient descent to local minima. The addition of isotropic noise can significantly slow down the convergence rate. In contrast, our setting is deterministic and corresponds to the simplest possible discrete-time implementation of gradient descent.

Numerous curvature-based optimization techniques have been developed in order to circumvent saddle points (e.g. trust-region methods
\cite{conn2000trust,sun2015complete}, modified Newton's method with curvilinear line search \cite{more1979use}, cubic regularized Newton's method \cite{nesterov2006cubic}, and saddle-free Newton methods \cite{dauphin2014identifying}). Unlike gradient descent, these methods have superlinear per-iteration implementation costs, making them impractical for high dimensional settings.

Gradient descent with carefully chosen initial conditions can bypass the problem of local minima altogether and converge to the global minimum for many practical non-convex optimization settings (e.g.  dictionary learning \cite{arora2015simple}, latent-variable models \cite{zhang2014spectral}, matrix completion \cite{keshavan2009matrix}, and phase retrieval \cite{candes2015phase}). In contrast, we focus on the performance of gradient descent under generic initial conditions.
Finally, some recent work has been focusing on the connections between stability and efficiency of fixed points in non-convex optimization (e.g., Gaussian random fields \cite{choromanska2014loss}).

Gradient-like dynamics, where the dynamic moves towards states of decreased cost but without necessarily moving in the direction of steepest decrease, is a generalization of gradient dynamics that arise in a number of applications including game theory and mathematical biology. Similar arguments about convergence to local minima for almost all initial conditions have been argued \cite{kleinberg2009multiplicative,PP16,ITCSMPP15} for (variants of) replicator dynamics and multiplicative weights update algorithms when applied to games where the incentives of all agents are closely aligned.\footnote{Such games are known as potential/congestion
games \cite{kleinberg2009multiplicative} and correspond to games where all agents act as if they share a common cost/potential function that they are trying to minimize.} From the perspective of biology and specifically evolution, (variants of) replicator/MWUA  \cite{PNAS2:Chastain16062014,meir2015sex} capture standard models of the evolution of the frequencies of different genotypes within a species (preferential survival of the fittest). By analyzing the properties of local minimum energy states  we can derive completely different conclusions about the long term system behavior (in terms e.g., of the resulting genetic diversity) from the ones that follow from analyzing all saddle points \cite{ITCSMPP15}. In fact, understanding the properties of local minima raises interesting computational complexity questions \cite{2014arXiv1411.6322M}. Finally,  examining the stability properties of equilibria can help us capture quantitatively the long term behavior of biologically inspired gradient-like systems even under time-evolving fitness landscapes \cite{2015arXiv151101409M}. Given the emergent overlapping interests between these areas and (non convex) optimization theory, it seems that novel opportunities for cross-fertilization between these research communities arise.

\subsection{Organization} 
In Section \ref{sec:prelims}, we introduce the notation and definitions used throughout the paper and state formally our main theorems. In Section \ref{sec:proofs}, we prove our results establishing the negligible probability of converging to saddle points, addressing the possibility of continuums of equilibria, forward-invariant subspaces,
and establishing an upper bound on the step-size. In Section \ref{sec:examples}, we produce several examples showcasing the effectiveness of our methods.  Finally, we conclude in Section \ref{sec:conclusion} by suggesting directions for future work.

\section{Preliminaries}\label{sec:prelims}

{\bf Notation:} We use boldface letters, e.g., $\vec{x}$, to denote column vectors. We denote by $\spec{A}, \norm[2]{A}$ the spectral radius and spectral norm of a symmetric matrix $A$ respectively. We also use $\norm[2]{\vec{x}}$ for the $\ell_2$ norm of vector $\vec{x}$. By $\nabla ^2 f(\vec{x})$ we denote the Hessian of a twice differentiable function $f : \mathcal{E} \to \mathbb{R}$, for some set $\mathcal{E} \subseteq \mathbb{R}^N$.

Assume a minimization problem of the form $\min_{\vec{x} \in \mathbb{R}^N} f(\vec{x})$ where $f : \mathbb{R}^N \to \mathbb{R}$ is a twice continuously differentiable function. Gradient descent is one of the most well-known algorithms (discrete dynamical system) to attack this generic optimization problem. It is defined by the equations below:
$$
\vec{x}_{k+1} = \vec{x}_k - \alpha \nabla f(\vec{x}_k),
$$

\noindent
or equivalently $\vec{x}_{k+1} = g(\vec{x}_k)$
with $g(\vec{x}) = \vec{x} - \alpha \nabla f(\vec{x}),$ $g: \mathbb{R}^N \to \mathbb{R}^N$ and $\alpha>0$.

It is easy to see that the fixed points of the dynamical system $\vec{x}_{k+1} = g(\vec{x}_k)$ are exactly the points $\vec{x}$ so that $\nabla f(\vec{x})=\vec{0}$, called \emph{critical points or equilibria}.
The set of local minima of $f$ is a subset of the set of critical points of $f$. These two sets do not coincide and this poses a serious obstacle for proving strong theoretical guarantees for gradient descent, since the dynamics may converge to a critical point which is not a local minimum, called a \emph{saddle point}.

Lee \etal \cite{DBLP:journals/corr/LeeSJR16} argue, under technical conditions which include the assumption of isolated critical points, that the set of initial conditions that converge to \emph{strict saddle points}  is a zero measure set (for definition of strict saddle, see Definition \ref{def:main}). The paper leaves as an open question whether the condition of isolated equilibria is necessary. We prove that
 the set of initial conditions that converge to a strict saddle point is a zero measure set
 even in the case of non-isolated critical points\footnote{Our arguments hence allow for cost functions $f$'s with uncountably many critical points.}. Furthermore, one of the conditions for $f$ is that $\nabla f$ is globally Lipschitz, which implies that the second derivative of $f$ is bounded, i.e., there exists a $\beta>0$ such that for all $\vec{x}$ we have $\norm[2]{\nabla^2 f(\vec{x})} \leq \beta$. However, even third degree polynomial functions are not globally Lipschitz. We provide a theorem which can circumvent this assumption as long as the domain $\mathcal{S}$ is \emph{forward or positively invariant} with respect to $g$, i.e., $g(\mathcal{S}) \subseteq \mathcal{S}$. Finally, we provide an easy upper bound on the step-size $\alpha$, via eigenvalue analysis of the Jacobian of $g$, i.e., $I - \alpha\nabla ^2 f(\vec{x})$.

Below we give some necessary definitions as appeared in Lee \etal \cite{DBLP:journals/corr/LeeSJR16}.

\begin{definition}\label{def:main}
\noindent
\begin{itemize}
  \item A point $\vec{x}^*$ is a critical point of $f$ if $\nabla f (\vec{x}^*)= \vec{0}$. We denote by $C = \{\vec{x}: \nabla f(\vec{x})=\vec{0}\}$ the set of critical points (can be uncountably many).
  \item A critical point $\vec{x}^*$ is isolated if there is a neighborhood $U$ around $\vec{x}^*$ and $\vec{x}^*$ is the only critical point
in $U$.\footnote{If the critical points are isolated then they are countably many or finite.} Otherwise is called non-isolated.
  \item A critical point $\vec{x}^*$ of $f$ is a saddle point if for all neighborhoods $U$ around $\vec{x}^*$ there are $\vec{y},\vec{z} \in U$  such that $f(\vec{z}) \leq f(\vec{x}^*) \leq f(\vec{y})$.
  \item A critical point $\vec{x}^*$ of $f$ is a strict saddle if $\lambda_{\min}(\nabla^2 f(\vec{x}^*))<0$ (minimum eigenvalue of matrix $\nabla^2 f(\vec{x}^{*})$ is negative).
  \item A set $\mathcal{S}$ is called forward or positively invariant with respect to some function $h : \mathcal{E} \to \mathbb{R}^N$ with $\mathcal{S} \subseteq \mathcal{E} \subseteq\mathbb{R}^N$ if $h(\mathcal{S}) \subseteq \mathcal{S}$.
\end{itemize}
\end{definition}

\subsection{Main Results}
In \cite{DBLP:journals/corr/LeeSJR16}, the steps of the proof of their result are the following: Under the regularity assumption that $\nabla f$ is globally Lipschitz with some Lipschitz constant $L$, Lee \etal are able to show that $g(\vec{x}) = \vec{x} - \alpha \nabla f(\vec{x})$ is a diffeomorphism for $\alpha<1/L$. Afterwards, using the center-stable manifold theorem (see theorem \ref{thm:manifold}), they show that the set of initial conditions so that $g$ converges to saddle points has measure zero under the assumption that the critical points are isolated. We generalize their result for non-isolated critical points, answering one of their open questions (see also the example in Section \ref{eg:continuums}, where there is a line of critical points).
\begin{theorem}\label{thm:main1} {\bf [Non-isolated]} Let $f: \mathbb{R}^N \to \mathbb{R}$ be a twice continuously differentiable function and $\sup_{\vec{x} \in \mathbb{R}^N} \norm[2]{\nabla^2 f(\vec{x})} \leq L < \infty$. The set of initial conditions $\vec{x} \in \mathbb{R}^N$ so that gradient descent with step-size $0<\alpha< 1/L$ converges to a strict saddle point is of (Lebesgue) measure zero, without the assumption that critical points are isolated.
\end{theorem}
We can prove a stronger version of the theorem above, circumventing the globally Lipschitz condition for domains which are forward invariant (see also the example in Section \ref{eg:positive}).
\begin{theorem}\label{thm:main2}{\bf [Non-isolated, forward invariant]} Let $f: \mathcal{S} \to \mathbb{R}^N$ be twice continuously differentiable in an open convex set $\mathcal{S} \subseteq \mathbb{R}^N$ and $\sup_{\vec{x} \in \mathcal{S}} \norm[2]{\nabla^2 f(\vec{x})} \leq L <\infty $. If $g(\mathcal{S}) \subseteq \mathcal{S}$ (where $g(\vec{x}) = \vec{x} - \alpha \nabla f(\vec{x})$) then the set of initial conditions $\vec{x} \in \mathcal{S}$ so that gradient descent with step-size $0<\alpha< 1/L$ converges to a strict saddle point is of (Lebesgue) measure zero, without the assumption that critical points are isolated.
\end{theorem}
Finally, via eigenvalue analysis of $I - \alpha\nabla ^2 f(\vec{x})$, we can find upper bounds on the step-size of gradient descent. A straightforward theorem is the following:
\begin{theorem}\label{thm:main3}{\bf [Upper bound on step-size]} Let $f$ be twice continuously differentiable in an open set $\mathcal{S} \subseteq \mathbb{R}^N$ and  $\mathcal{C}^*$ be the set of local minima. Assume also that $\gamma  < \inf_{\vec{x} \in \mathcal{C}^*} \norm[2]{\nabla^2 f(\vec{x})}  < \infty$. A necessary condition so that gradient descent converges to local minima for all but (Lebesgue) measure zero initial conditions in $\mathcal{S}$ is that the step-size satisfies $\alpha < \frac{2}{\gamma}$.
\end{theorem}


\section{Proving the theorems}\label{sec:proofs}

Before we proceed with the proofs, let us argue that Theorem \ref{thm:main2} is a generalization of Theorem \ref{thm:main1}. This can be checked  by setting $\mathcal{S}  \defeq \mathbb{R}^N$ and observing that $g(\mathbb{R}^N) \subseteq \mathbb{R}^N$. We continue with the proofs of Theorems \ref{thm:main2} and \ref{thm:main3}.

\subsection{Proof of Theorem \ref{thm:main2}}

In this section, we prove Theorem \ref{thm:main2}. We start by showing that the assumptions of Theorem \ref{thm:main2} imply that $\nabla f(\vec{x})$ is Lipschitz in $\mathcal{S}$.
\begin{lemma} \label{lem:lipsc}Let $f : \mathcal{S} \to \mathbb{R}^N$ where $\mathcal{S}$ is an open convex set and $f$ be twice continuously differentiable in $\mathcal{S}$. Also assume that $\sup_{\vec{x}\in \mathcal{S}}\norm[2]{\nabla^2 f(\vec{x})} \leq L <\infty$. Then $\nabla f$ satisfies the Lipschitz condition in $\mathcal{S}$ with Lipschitz constant $L$.
\end{lemma}
\begin{proof} Let $\vec{x},\vec{y} \in \mathcal{S}$ (column vectors) and define the function $H:[0,1] \to \mathbb{R}^N $ as $H(t) = \nabla f(\vec{x}+t(\vec{y}-\vec{x}))$. By the chain rule we get that $H'(t) \defeq \frac{d H}{dt} = (\nabla ^2 f( \vec{x} + t(\vec{y}-\vec{x})))\cdot(\vec{y}-\vec{x})$.
It holds that
\begin{align*}
\norm[2]{\nabla f(\vec{y}) - \nabla f(\vec{x})} = \norm[2]{\int_{0}^1 H'(t) dt} &\leq \int_0^1 \norm[2]{H'(t)}dt
\\& = \int_{0}^1 \norm[2]{(\nabla ^2 f( \vec{x} + t(\vec{y}-\vec{x})))(\vec{y}-\vec{x})} dt
\\& \leq \int_{0}^1 \norm[2]{\nabla ^2 f( \vec{x} + t(\vec{y}-\vec{x}))}\norm[2]{\vec{y}-\vec{x}} dt
\\& \leq \int_{0}^1 L \norm[2]{\vec{y}-\vec{x}} dt = L \norm[2]{\vec{y}-\vec{x}}.
\end{align*}
\end{proof}
\begin{remark}\label{rem:spec} From Schwarz's theorem we get that $\nabla^2 f(\vec{x})$ is symmetric for $\vec{x} \in \mathcal{S}$, hence $\norm[2]{\nabla^2 f(\vec{x})} = \spec{\nabla^2 f(\vec{x})}$.
\end{remark}
The assumption that $\sup_{\vec{x} \in \mathcal{S}}\norm[2]{\nabla^2 f(\vec{x})} \leq L <\infty$ implies that $\nabla f(x)$ is Lipschitz with constant $L$ in the convex set $\mathcal{S}$, as stated by Lemma \ref{lem:lipsc}. We show that the converse holds as well, i.e., the Lipschitz condition for $\nabla f(\vec{x})$ with constant $L$ in the main theorem in Lee \etal implies $\norm[2]{\nabla^2 f(\vec{x})} \leq L$ for all $\vec{x} \in \mathcal{S}$ and hence the assumption in our Theorems \ref{thm:main1}, \ref{thm:main2} that $\sup_{\vec{x} \in \mathcal{S}} \norm[2]{\nabla^2 f(\vec{x})} \leq L$ is satisfied.
\begin{lemma}\label{lem:invlipsc} Let $f : \mathcal{S} \to \mathbb{R}^N$ where $\mathcal{S}$ is an open convex set and $f$ is twice continuously differentiable in $\mathcal{S}$. Assume $\nabla f(\vec{x})$ is Lipschitz with constant $L$ in $\mathcal{S}$ then it holds $\sup_{\vec{x} \in \mathcal{S}}\norm[2]{\nabla^2 f(\vec{x})} \leq L$.
\end{lemma}
\begin{proof} Fix an $\epsilon>0$. By Taylor's theorem since $f$ is twice differentiable with respect to some point $\vec{x}$ it holds that
\begin{align*}
 \norm[2]{\nabla f(\vec{y})- \nabla f(\vec{x})} &\geq \norm[2]{(\nabla^2 f(\vec{x}))(\vec{y} - \vec{x})} - o(\norm[2]{\vec{y}-\vec{x}})
\\&\geq  \norm[2]{(\nabla^2 f(\vec{x}))(\vec{y} - \vec{x})} - \epsilon \norm[2]{\vec{y}-\vec{x}}
\end{align*}
for $\vec{y}$ sufficiently close to $\vec{x}$ (depends on $\epsilon$). Therefore under the Lipschitz assumption we get that there exists a closed neighborhood $U(\epsilon)$ of $\vec{x}$ so that for all $\vec{y} \in U$ we get
\begin{equation}\label{eq:inequality}
   \norm[2]{(\nabla^2 f(\vec{x}))(\vec{x} - \vec{y})} \leq \norm[2]{\nabla f(\vec{x})- \nabla f(\vec{y})} + \epsilon \norm[2]{\vec{y}-\vec{x}} \leq (L+\epsilon) \norm[2]{\vec{x}-\vec{y}}.
\end{equation}
We consider a closed ball $B$ subset of $U$, with center $\vec{x}$ and radius $r$ (in $\ell_{2}$) and set $\vec{z} = \vec{x} - \vec{y}$. It is true that $\norm[2]{\nabla^2 f(\vec{x})} = \sup_{\norm[2]{\vec{z}} = r} \frac{\norm[2]{(\nabla^2 f(\vec{x})) \vec{z}}}{\norm[2]{\vec{z}}}$ by definition of spectral norm, scaled so that the length of the vectors is exactly $r$. Using \ref{eq:inequality} we get that $\norm[2]{\nabla^2 f(\vec{x})} \leq L+\epsilon$. Since $\epsilon$ is arbitrary, we get that $\norm[2]{\nabla^2 f(\vec{x})} \leq L$. We conclude that $\sup_{\vec{x} \in \mathcal{S}}\norm[2]{\nabla^2 f(\vec{x})} \leq L$.
\end{proof}
Lemmas \ref{lem:lipsc} and \ref{lem:invlipsc} show that the smoothness assumptions in Lee \etal paper are equivalent to ours. We use the condition on the spectral norm of the matrix $\nabla^2 f(\vec{x})$ so that we can work with the eigenvalues in our theorems (e.g., in Remark \ref{rem:spec} the spectral norm coincides with spectral radius for $\nabla^2 f(\vec{x})$).
Below we prove that the update rule of gradient descent, i.e., function $g$ is a diffeomorphism under the assumptions of Theorem \ref{thm:main2} (similar approach appeared in \cite{DBLP:journals/corr/LeeSJR16}).
\begin{lemma} \label{lem:diffeo}Under the assumptions of Theorem \ref{thm:main2}, function $g$ is a diffeomorphism in $\mathcal{S}$.
\end{lemma}
\begin{proof} First we prove that $g$ is a injective. We follow the same argument as in \cite{DBLP:journals/corr/LeeSJR16}. Suppose $g(\vec{y}) = g(\vec{x})$, thus $\vec{y} - \vec{x} = \alpha(\nabla f(\vec{y}) - \nabla f(\vec{x}))$. We assume that $\vec{x} \neq \vec{y}$ and we will reach contradiction. From Lemma \ref{lem:lipsc} we get $\norm[2]{\nabla f(\vec{y})- \nabla f(\vec{x})} \leq L \norm[2]{\vec{y}-\vec{x}}$ and hence $\norm[2]{\vec{x} - \vec{y}} \leq \alpha L \norm[2]{\vec{y}-\vec{x}}< \norm[2]{\vec{y}-\vec{x}}$ since $\alpha L <1$ (contradiction).

We continue by showing that $g$ is a local diffeomorphism. Observe that the Jacobian of $g$ is $I - \alpha \nabla^2 f(\vec{x})$. It suffices to show that $\alpha \nabla^2 f(\vec{x})$ has no eigenvalue which is 1, because this implies matrix $I - \alpha \nabla^2 f(\vec{x})$ is invertible. As long as $I - \alpha \nabla^2 f(\vec{x})$ is invertible, from Inverse Function Theorem (see \cite{spivak}) follows that $g$ is a local diffeomorphism. Finally, since $g$ is injective, the inverse $g^{-1}$ is well defined and since $g$ is a local diffeomorphism in $\mathcal{S}$, it follows that $g^{-1}$ is smooth in $\mathcal{S}$. Therefore $g$ is a diffeomorphism.

Let $\lambda$ be an eigenvalue of $\nabla^2 f(\vec{x})$. Then $|\lambda| \leq \spec{\nabla^2 f(\vec{x})} = \norm[2]{\nabla^2 f(\vec{x})} \leq L$ where the equality comes from Remark \ref{rem:spec} and first and last inequalities are satisfied by assumption. Therefore $\alpha\nabla^2 f(\vec{x})$ has as eigenvalue $\alpha \lambda$ and $|\alpha \lambda| \leq \alpha L < 1$. Thus all eigenvalues of $\alpha \nabla^2 f(\vec{x})$ are less than 1 in absolute value and the proof is complete.
\end{proof}

We will use the center-stable manifold theorem since $g(\vec{x}) = \vec{x}-\alpha\nabla f(\vec{x})$ is a diffeomorphism, where $\sup_{\vec{x} \in \mathcal{S}}\norm[2]{\nabla^2 f(\vec{x})} \leq L$ and $\alpha < 1/L$.  A modification of this proof for replicator dynamics (not gradient descent) appeared in \cite{PP16} and \cite{ITCSMPP15}.

\begin{theorem}[Center and Stable Manifolds, p. 65 of \cite{shub}]\label{thm:manifold}
Let $\vec{p}$ be a fixed point for the $C^r$ local diffeomorphism $h: U \to \mathbb{R}^n$ where $U \subset \mathbb{R}^n$ is an open neighborhood of $\vec{p}$ in $\mathbb{R}^n$ and $r \geq 1$. Let $E^s \oplus E^c \oplus E^u$ be the invariant splitting
of $\mathbb{R}^n$ into generalized eigenspaces of $Dh(\vec{p})$\footnote{Jacobian of $h$ evaluated at $\vec{p}$.} corresponding to
eigenvalues of absolute value less than one, equal to one, and greater than one. To the $Dh(\vec{p})$ invariant subspace $E^s\oplus
E^c$ there is an associated local $h$ invariant $C^r$ embedded disc $W^{sc}_{loc}$ of dimension $dim(E^s \oplus E^c)$, and ball $B$ around $\vec{p}$ such that:
\begin{equation} h(W^{sc}_{loc}) \cap B \subset W^{sc}_{loc}.\textrm{  If } h^n(\vec{x}) \in B \textrm{ for all }n \geq 0,
\textrm{ then }\vec{x} \in W^{sc}_{loc}.
\end{equation}
\end{theorem}

From this point on
our approach deviates significantly from that of \cite{DBLP:journals/corr/LeeSJR16}  and new, orthogonal ideas and tools need to be introduced. 
\medskip

Let $\vec{r}$ be a critical point of function $f(\vec{x})$ and $B_{\vec{r}}$ be the (open) ball that is derived from Theorem \ref{thm:manifold}. We consider the union of these balls $$A = \cup _{\vec{r}}B_{\vec{r}}.$$ The following property for $\mathbb{R}^N$ holds:

\begin{theorem} [Lindel\H{o}f's lemma \cite{kelley}] For every open cover there is a countable subcover.
\end{theorem}

Therefore due to Lindel\H{o}f's lemma, we can find a countable subcover for $A$, i.e., there exists fixed-points $\vec{r}_1,\vec{r}_2,\dots$ such that $A = \cup _{m=1}^{\infty}B_{\vec{r}_{m}}$.
If  gradient descent converges to a strict saddle point, starting from a point $\vec{v} \in \mathcal{S}$, there must exist a $t_{0}$ and $m$ so that $g^{t}(\vec{v}) \in B_{\vec{r}_{m}}$ for all $t \geq t_{0}$. From Theorem \ref{thm:manifold} we get that $g^t(\vec{v}) \in W_{loc}^{sc}(\vec{r}_m) \cap \mathcal{S}$ where we used the fact that $g(\mathcal{S}) \subseteq \mathcal{S}$ (from assumption forward invariant), namely the trajectory remains in $\mathcal{S}$ for all times \footnote{$W_{loc}^{sc}(\vec{r}_m)$ denotes the center stable manifold of fixed point $\vec{r}_m$}. By setting $D_{1}(\vec{r}_m) = g^{-1}(W_{loc}^{sc}(\vec{r}_m) \cap \mathcal{S})$ and $D_{i+1}(\vec{r}_m) = g^{-1}(D_{i}(\vec{r}_m) \cap \mathcal{S})$ we get that $\vec{v} \in D_{t}(\vec{r}_m)$ for all $t \geq t_0$.
Hence the set of initial points in $\mathcal{S}$ so that gradient descent converges to a strict saddle point is a subset of
\begin{equation}
P =  \cup_{m=1}^{\infty} \cup_{t=0}^{\infty} D_{t}(\vec{r}_m)).
\end{equation}

Since $\vec{r}_m$ is strict saddle point, the Jacobian $I - \alpha \nabla^2 f(\vec{x})$ has an eigenvalue greater than 1, namely the dimension of the unstable eigenspace satisfies $dim(E^u)\ge 1$, and therefore the dimension of $W_{loc}^{sc}(\vec{r}_m)$ is at most $N-1$. Thus,
the set $W_{loc}^{sc}(\vec{r}_m) \cap \mathcal{S}$ has Lebesgue measure zero in $\mathbb{R}^N$. 
 Finally since $g$ is a diffeomorphism (from Lemma \ref{lem:diffeo}), $g^{-1}$ is continuously differentiable and thus it is locally Lipschitz (see \cite{perko} p.71). Therefore using Lemma \ref{lips} below, $g^{-1}$ preserves the null-sets and hence (by induction) $D_{i}(\vec{r}_m)$ has measure zero for all $i$. Thereby we get that $P$ is a countable union of measure zero sets, i.e., is measure zero as well and the claim of Theorem \ref{thm:main2} follows. \qed

\noindent
 The following lemma is standard, but we provide a proof for completeness.

\begin{lemma}\label{lips} Let $h: \mathcal{S} \to \mathbb{R}^m$ be a locally Lipschitz function with $\mathcal{S} \subseteq \mathbb{R}^m$ then $h$ is null-set preserving, i.e., for $E \subset \mathcal{S}$ if $E$ has measure zero then $h(E)$ has also measure zero.
\end{lemma}
\begin{proof} Let $B_{\gamma}$ be an open ball such that $\norm[]{h(\vec{y}) - h(\vec{x})} \leq K_{\gamma} \norm[]{\vec{y}-\vec{x}}$
for all $\vec{x},\vec{y} \in B_{\gamma}$. We consider the union $\cup_{\gamma}B_{\gamma}$ which cover $\mathbb{R}^m$ by the
assumption that $h$ is locally Lipschitz. By Lindel\H{o}f's lemma we have a countable subcover, i.e., $\cup_{i=1}^{\infty}B_{i}$. Let
$E_{i} = E \cap B_{i}$. We will prove that $h(E_{i})$ has measure zero. Fix an $\epsilon >0$. Since $E_{i} \subset E$, we have that
$E_{i}$ has measure zero, hence we can find a countable cover of open balls $C_{1},C_{2},... $ for $E_{i}$, namely $E_{i} \subset
\cup_{j=1}^{\infty}C_{j}$ so that $C_{j} \subset B_{i}$ for all $j$ and also $\sum_{j=1}^{\infty} \mu(C_{j}) <
\frac{\epsilon}{K_{i}^m}$. Since $E_{i} \subset \cup_{j=1}^{\infty}C_{j}$ we get that $h(E_{i}) \subset \cup_{j=1}^{\infty}h(C_{j})$,
namely $h(C_{1}),h(C_{2}),...$ cover $h(E_{i})$ and also $h(C_{j}) \subset h(B_{i})$ for all $j$. Assuming that ball $C_{j} \equiv
B(\vec{x},r)$ (center $\vec{x}$ and radius $r$) then it is clear that $h(C_{j}) \subset B(h(\vec{x}),K_{i} r)$ ($h$ maps the
center $\vec{x}$ to $h(\vec{x})$ and the radius $r$ to $K_{i}r$ because of Lipschitz assumption). But $\mu(B(h(\vec{x}),K_{i}
r)) = K_{i}^m \mu(B(\vec{x}, r)) = K_{i}^m \mu(C_{j})$, therefore $\mu(h(C_{j})) \leq K_{i}^m \mu(C_{j})$ and so we conclude that
$$\mu(h(E_{i})) \leq \sum_{j=1}^{\infty}\mu(h(C_{j})) \leq K_{i}^m \sum_{j=1}^{\infty}\mu(C_{j}) < \epsilon$$ Since $\epsilon$ was
arbitrary, it follows that $\mu(h(E_{i})) =0$. To finish the proof, observe that $h(E) = \cup_{i=1}^{\infty} h(E_{i})$ therefore
$\mu(h(E)) \leq \sum_{i=1}^{\infty} \mu(h(E_{i})) =0$.  \end{proof}

\noindent
A straightforward application of Theorem \ref{thm:main2} is the following:
\begin{corollary}\label{cor:one}
Assume that the conditions of Theorem \ref{thm:main2} are satisfied and all saddle points of $f$ are strict. Additionally, let $\nu$ be a prior measure with support $\mathcal{S}$ which is absolutely continuous with respect to Lebesgue measure, and assume $\lim_{k \to \infty} g^k(x)$ exists\footnote{$g^k$ denotes the composition of $g$ with itself $k$ times.} for all $\vec{x}$ in $\mathcal{S}$. Then $$\mathbb{P}_{\nu}[\lim_k g^k(\vec{x}) = \vec{x}^*]=1,$$ where $\vec{x}^*$ is a local minimum.
\end{corollary}
\begin{proof}
Since the set of initial conditions whose limit point is a (strict) saddle point is a measure zero set and we have assumed $\lim_{k \to \infty} g^k(x)$ exists for all initial conditions in $\mathcal{S}$ then the probability of converging to a local minimizer is $1$.
\end{proof}
\begin{remark}
Arguing that $\lim_k g^k(\vec{x})$ exists follows from standard arguments in several settings of interest (e.g for analytic functions $f$ that satisfy (Lojasiewicz Gradient Inequality)), see paper \cite{DBLP:journals/corr/LeeSJR16} and references therein.
\end{remark}
The importance of Theorem \ref{thm:main2} will become clear in the examples of Section \ref{sec:examples}. Specifically, in the example of Section \ref{eg:positive}, the function is not globally Lipschitz (we use the example that appears in \cite{DBLP:journals/corr/LeeSJR16}), nevertheless Theorem \ref{thm:main2} applies and thus we have convergence to local minimizers with probability 1. In the example of Section \ref{eg:continuums} we see that simple functions may have non-isolated critical points.
\subsection{Proof of Theorem \ref{thm:main3}}

\begin{proof}
We proceed by contradiction. Consider any local minimum $\vec{x}^*$, and by assumption we get that $\spec{\nabla^2 f(\vec{x}^*)}> \gamma$. Let $\alpha \geq \frac{2}{\gamma}$. Therefore the Jacobian $I - \alpha\nabla^2 f(\vec{x}^*)$ of $g$ at $\vec{x}^*$ has spectral radius greater than 1 since $\spec{I - \alpha\nabla^2 f(\vec{x}^*)} \geq \spec{\alpha\nabla^2 f(\vec{x}^*)}-1 > \alpha \gamma -1 \geq 1$. This implies that the fixed point $\vec{x}^*$ of $g$ is (Lyapunov) unstable. Since this is true for every local minimum, it cannot be true that gradient descent converges with probability 1 to local minima. 
\end{proof}


\section{Examples}\label{sec:examples}

\subsection{Example for non-isolated critical points}\label{eg:continuums}

Consider the simple example of the cost function $f:\R^3\rightarrow \R$ with $f(x,y,z)= 2xy+2xz -2x -y -z$. Its gradient is $\nabla(f)=(2y+2z-2,2x-1,2x-1)$. Naturally, its saddle points correspond exactly to the line $(1/2,w,1-w)$ for $w\in \R$ and by computing their (common) eigenvalues we establish that they are all strict saddles (their minimum eigenvalue is $-2\sqrt{2}$). As we expect from our analysis effectively no trajectories converge to them (instead the value of practically all trajectories goes to $-\infty$). We plot in red some sample trajectories for small enough step sizes, starting in the local neighborhood of the equilibrium set.

\begin{figure}[!htb]
\centering
\includegraphics[width=0.4\linewidth]{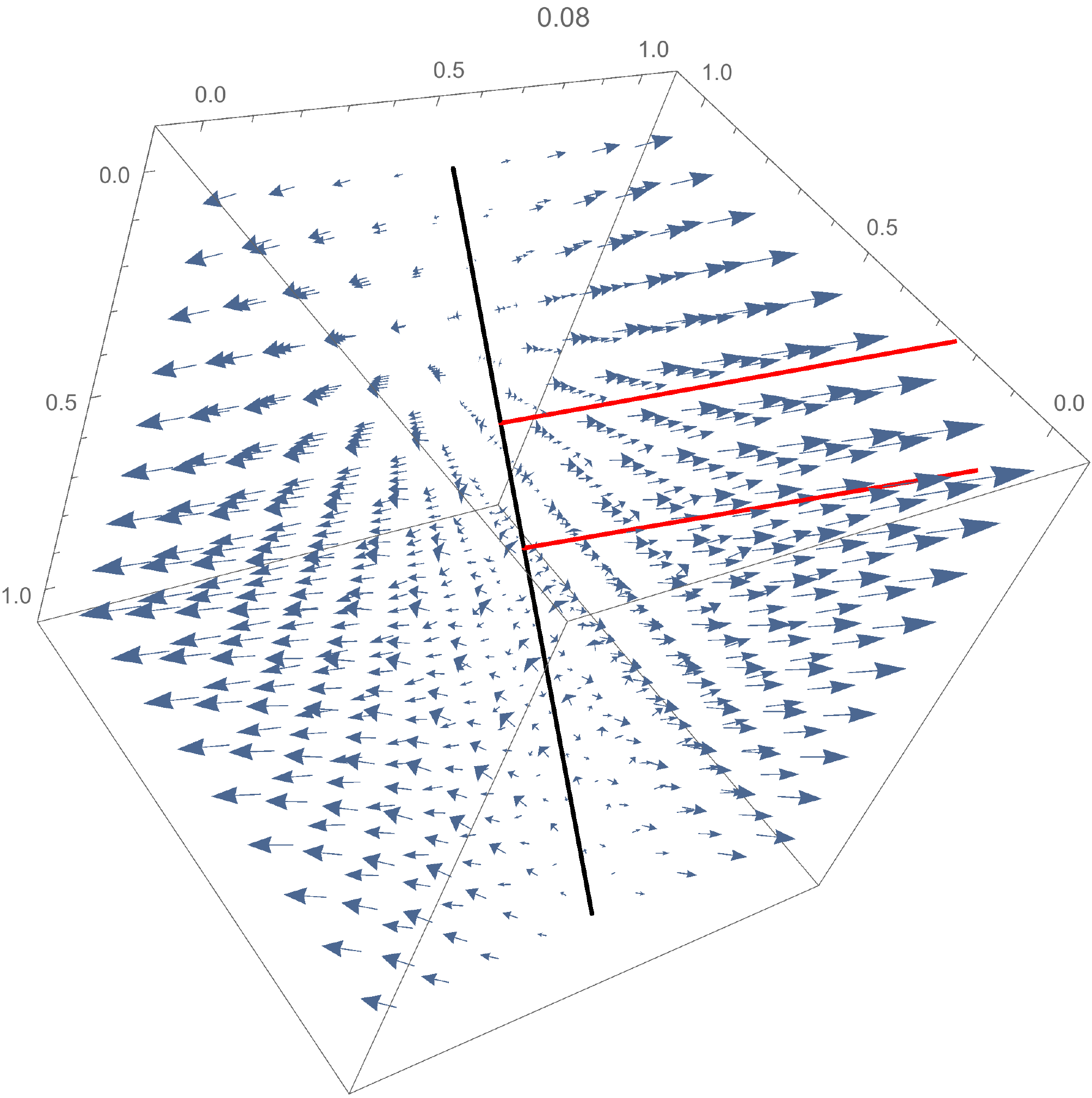}
\caption{Example that satisfies the assumptions of Theorem \ref{thm:main1}. The black line represent critical points of $f$, all of which are strict. The red lines correspond to diverging trajectories of gradient descent with small step size.}
\end{figure}

\subsection{Example for forward invariant set}\label{eg:positive} We use the same function as in Lee \etal $f(x,y) = \frac{x^2}{2} + \frac{y^4}{4} - \frac{y^2}{2}$. As argued in previous sections, $f$ is not globally Lipschitz so the main result in \cite{DBLP:journals/corr/LeeSJR16} cannot be applied here. We will use our Theorem \ref{thm:main2} which talks about forward invariant domains.

The critical points of $f$ are $(0,0), (0,1), (0,-1)$. $(0,0)$ is a strict saddle point and the other two are local minima. Observe that the Hessian $\nabla^2 f(x,y)$ is \[J = \left( \begin{array}{cc}
1 & 0 \\
0 & 3y^2-1\end{array} \right).\] For $\mathcal{S} = (-1,1) \times (-2,2)$, so we get that $\sup_{(x,y) \in \mathcal{S}} \norm[2]{\nabla^2 f(x,y)} \leq 11$ (for $y=2$ gets the maximum value). We choose $\alpha = \frac{1}{12} < \frac{1}{11}$, and we have $g(x,y) =  ((1-\alpha)x, (1+\alpha)y-\alpha y^3) = (\frac{11x}{12}, \frac{13y}{12} - \frac{y^3}{12})$. It is not difficult to see that $g(\mathcal{S}) \subseteq \mathcal{S}$ (easy calculations). The assumptions of Theorem \ref{thm:main2} are satisfied, hence it is true that the set of initial conditions in $\mathcal{S}$ so that gradient descent converges to $(0,0)$ has measure zero. Moreover, by Corollary \ref{cor:one} it holds that if the initial condition is taken (say) uniformly at random in $\mathcal{S}$, then gradient descent converges to $(0,1),(0,-1)$ with probability 1. The figure below makes the claim clear, i.e. the set of initial conditions so that gradient descent converges to $(0,0)$ lie on the axis $y=0$, which is of measure zero in $\mathbb{R}^2$. For all other starting points, gradient descent converges to local minima. Finally, from the figure one can see that $\mathcal{S}$ is forward invariant.
\begin{figure}[!htb]
\centering
\includegraphics[width=0.4\linewidth]{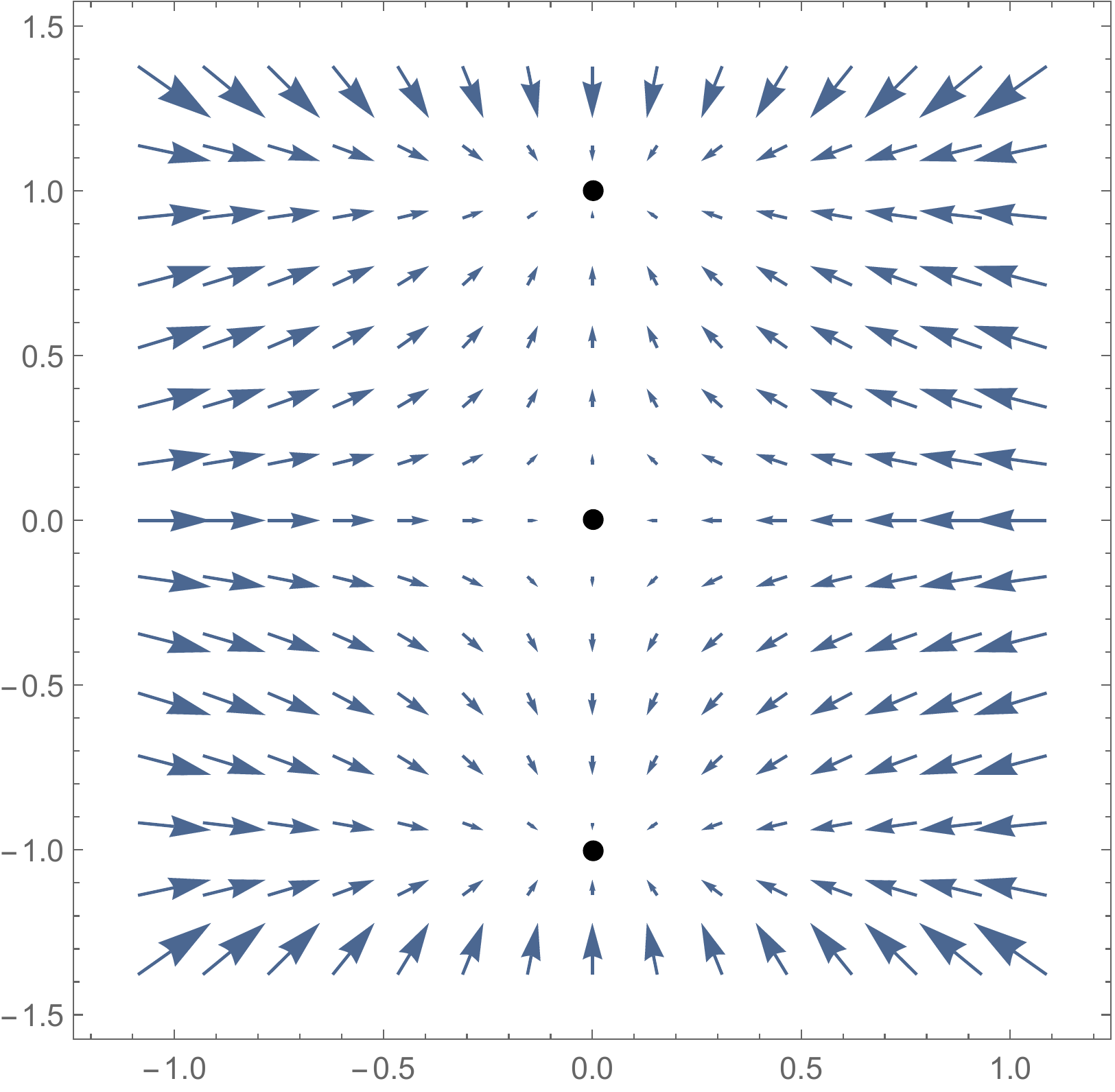}
\caption{Example that satisfies the assumptions of Theorem \ref{thm:main2}. The three black dots represent the critical points. Function $f$ is not Lipschitz.}
\end{figure}

\subsection{Example for step-size}\label{eg:stepsize}
We use the same function as in the previous example. Observe that for $(0,0), (0,1),(0,-1)$ we have that the spectral radius of $\nabla^2 f$ is $1,2,2$ respectively (so the minimum of all is 1). We choose $\alpha \geq 2$ and we get that $g(x,y) = (-x,3y - 2y^3)$. It is not hard to see that
gradient descent does not converge (in the first coordinate function $g$ cycles between $x$ and $-x$).
\section{Conclusion}\label{sec:conclusion}

Our work argues that saddle points are indeed of little concern for the gradient descent method in practice under rather weak assumptions for $f$ which allow for non-isolated critical points. In some sense, this is the strongest positive result possible without making explicit assumptions on the structure of the cost function $f$ nor using beneficial random noise/well chosen initial conditions. Naturally, all these directions are of key interest and are the object of recent work (see section \ref{sec:related}). Keeping up with this simplest, deterministic implementation of gradient descent a natural hypothesis is that (in settings of practical interest) it converges not only to local minimizers but moreover the size of the region of attraction of each local minimizer is in a sense directly proportional to its quality. 

Recently, in \cite{PP16} there has been some progress in proving such statements in non-convex gradient-like systems that arise from learning in games. In such settings, (stable) fixed points correspond to Nash equilibria, but instead of having the typical system performance being dominated by the worst case Nash equilibria (as Price of Anarchy suggests) the regions of attractions of such bad (social) states prove to be minimal and the system works near optimally on average (given uniformly random initial conditions). Extending such statements to actual gradient-dynamics as well as comparing the average case performance of different heuristics even in restricted settings is a fascinating question that could shed more light into the in-many-cases surprising efficiency of the gradient descent method.

\section*{Acknowledgements}\label{sec:ack}

We are grateful to Jason D. Lee, Max Simchowitz, Michael I. Jordan and Benjamin Recht 
 for their support and helpful discussions and suggestions as well as for the elucidating blog article at \url{http://www.offconvex.org} on their elegant work, which inspired our investigation. We are also thankful to Nisheeth Vishnoi, on whose blog the article appeared, for pointing it out to us.

\bibliographystyle{plain}
\bibliography{sigproc3}

\begin{thebibliography}{10}

\bibitem{arora2015simple}
Sanjeev Arora, Rong Ge, Tengyu Ma, and Ankur Moitra.
\newblock Simple, efficient, and neural algorithms for sparse coding.
\newblock In {\em 28th Conference on Learning Theory}, pages 113--149, 2015.

\bibitem{candes2015phase}
Emmanuel~J Candes, Xiaodong Li, and Mahdi Soltanolkotabi.
\newblock Phase retrieval via wirtinger flow: Theory and algorithms.
\newblock {\em Information Theory, IEEE Transactions on}, 61(4):1985--2007,
  2015.

\bibitem{PNAS2:Chastain16062014}
Erick Chastain, Adi Livnat, Christos Papadimitriou, and Umesh Vazirani.
\newblock Algorithms, games, and evolution.
\newblock {\em Proceedings of the National Academy of Sciences (PNAS)},
  111(29):10620--10623, 2014.

\bibitem{choromanska2014loss}
Anna Choromanska, Mikael Henaff, Michael Mathieu, G{\'e}rard~Ben Arous, and
  Yann LeCun.
\newblock The loss surfaces of multilayer networks.
\newblock {\em arXiv preprint arXiv:1412.0233}, 2014.

\bibitem{conn2000trust}
Andrew~R Conn, Nicholas~IM Gould, and Ph~L Toint.
\newblock {\em Trust region methods}, volume~1.
\newblock Siam, 2000.

\bibitem{dauphin2014identifying}
Yann~N Dauphin, Razvan Pascanu, Caglar Gulcehre, Kyunghyun Cho, Surya Ganguli,
  and Yoshua Bengio.
\newblock Identifying and attacking the saddle point problem in
  high-dimensional non-convex optimization.
\newblock In {\em Advances in neural information processing systems}, pages
  2933--2941, 2014.

\bibitem{ge2015escaping}
Rong Ge, Furong Huang, Chi Jin, and Yang Yuan.
\newblock Escaping from saddle points---online stochastic gradient for tensor
  decomposition.
\newblock {\em arXiv preprint arXiv:1503.02101}, 2015.

\bibitem{kelley}
John~L. Kelley.
\newblock {\em General Topology}.
\newblock Springer, 1955.

\bibitem{keshavan2009matrix}
Raghunandan~H Keshavan, Sewoong Oh, and Andrea Montanari.
\newblock Matrix completion from a few entries.
\newblock In {\em Information Theory, 2009. ISIT 2009. IEEE International
  Symposium on}, pages 324--328. IEEE, 2009.

\bibitem{kleinberg2009multiplicative}
Robert Kleinberg, Georgios Piliouras, and Eva Tardos.
\newblock Multiplicative updates outperform generic no-regret learning in
  congestion games.
\newblock In {\em STOC}, pages 533--542. ACM, 2009.

\bibitem{lee2001algorithms}
Daniel~D Lee and H~Sebastian Seung.
\newblock Algorithms for non-negative matrix factorization.
\newblock In {\em Advances in neural information processing systems}, pages
  556--562, 2001.

\bibitem{DBLP:journals/corr/LeeSJR16}
Jason~D. Lee, Max Simchowitz, Michael~I. Jordan, and Benjamin Recht.
\newblock Gradient descent only converges to minimizers.
\newblock {\em Accepted to Appear in COLT, CoRR}, abs/1602.04915, 2016.

\bibitem{ITCSMPP15}
Ruta Mehta, Ioannis Panageas, and Georgios Piliouras.
\newblock Natural selection as an inhibitor of genetic diversity:
  Multiplicative weights updates algorithm and a conjecture of haploid
  genetics.
\newblock In {\em Innovations in Theoretical Computer Science, {ITCS}}, 2015.

\bibitem{2015arXiv151101409M}
Ruta Mehta, Ioannis Panageas, Georgios Piliouras, Prasad Tetali, and Vijay~V.
  Vazirani.
\newblock {Mutation, Sexual Reproduction and Survival in Dynamic Environments}.
\newblock {\em ArXiv e-prints}, http://arxiv.org/abs/1511.01409, 2015.

\bibitem{2014arXiv1411.6322M}
Ruta {Mehta}, Ioannis {Panageas}, Georgios {Piliouras}, and Sadra {Yazdanbod}.
\newblock {The Complexity of Genetic }.
\newblock {\em ArXiv e-prints}, http://arxiv.org/abs/1411.6322, 2014.

\bibitem{meir2015sex}
Reshef Meir and David Parkes.
\newblock On sex, evolution, and the multiplicative weights update algorithm.
\newblock In {\em Proceedings of the 2015 International Conference on
  Autonomous Agents and Multiagent Systems}, pages 929--937. International
  Foundation for Autonomous Agents and Multiagent Systems, 2015.

\bibitem{more1979use}
Jorge~J Mor{\'e} and Danny~C Sorensen.
\newblock On the use of directions of negative curvature in a modified newton
  method.
\newblock {\em Mathematical Programming}, 16(1):1--20, 1979.

\bibitem{Nesterov04}
Yurii Nesterov.
\newblock {\em Introductory lectures on convex optimization}, volume~87.
\newblock Springer Science and Business Media, 2004.

\bibitem{nesterov2006cubic}
Yurii Nesterov and Boris~T Polyak.
\newblock Cubic regularization of newton method and its global performance.
\newblock {\em Mathematical Programming}, 108(1):177--205, 2006.

\bibitem{PP16}
Ioannis Panageas and Georgios Piliouras.
\newblock Average case performance of replicator dynamics in potential games
  via computing regions of attraction.
\newblock {\em 17th ACM Conference on Economics and Computation (EC)},
  http://arxiv.org/abs/1403.3885, 2016.

\bibitem{pemantle1990nonconvergence}
Robin Pemantle.
\newblock Nonconvergence to unstable points in urn models and stochastic
  approximations.
\newblock {\em The Annals of Probability}, pages 698--712, 1990.

\bibitem{perko}
Lawrence Perko.
\newblock {\em Differential Equations and Dynamical Systems}.
\newblock Springer, 3nd. edition, 1991.

\bibitem{ravindran2006engineering}
A~Ravindran, Gintaras~Victor Reklaitis, and Kenneth~Martin Ragsdell.
\newblock {\em Engineering optimization: methods and applications}.
\newblock John Wiley \& Sons, 2006.

\bibitem{sandholm2009evolutionary}
William~H Sandholm.
\newblock Evolutionary game theory.
\newblock In {\em Encyclopedia of Complexity and Systems Science}, pages
  3176--3205. Springer, 2009.

\bibitem{shub}
Michael Shub.
\newblock {\em Global Stability of Dynamical Systems}.
\newblock Springer-Verlag, 1987.

\bibitem{spivak}
Michael Spivak.
\newblock {\em Calculus On Manifolds: A Modern Approach To Classical Theorems
  Of Advanced Calculus}.
\newblock Addison-Wesley, 1965.

\bibitem{sun2015complete}
Ju~Sun, Qing Qu, and John Wright.
\newblock Complete dictionary recovery over the sphere ii: Recovery by
  riemannian trust-region method.
\newblock {\em arXiv preprint arXiv:1511.04777}, 2015.

\bibitem{zhang2014spectral}
Yuchen Zhang, Xi~Chen, Denny Zhou, and Michael~I Jordan.
\newblock Spectral methods meet em: A provably optimal algorithm for
  crowdsourcing.
\newblock In {\em NIPS}, pages 1260--1268, 2014.

\end{thebibliography}

\end{document}